\documentclass{amsart}
\usepackage{ifpdf}
\ifpdf

\usepackage[pdftex]{graphicx}
\usepackage[pdftex,pdfstartview=FitBH,
  bookmarks=true, bookmarksopen=true,
  nesting=false,hyperindex=false,
  backref=false, raiselinks=false,
  colorlinks=false]{hyperref}
\else
\usepackage{graphicx}
\usepackage[bookmarks=true, bookmarksopen=true,
  nesting=false,hyperindex=false,
  backref=false, raiselinks=false,
  colorlinks=false]{hyperref}
\fi
\usepackage[utf8]{inputenc} 
\usepackage{xcolor}

\providecommand{\ifdprodangle}{\iffalse} 
\newcommand{\hylabel}[1]{\ifpdf\hypertarget{#1_l}{}\label{#1}\else\label{#1}\fi}

\newcommand{\hyautoref}[1]{\ifpdf\hyperlink{#1_l}{\autoref{#1}}\else\autoref{#1}\fi}

\newcommand{\includegraphix}[1]{\ifpdf\includegraphics{#1.pdf}\else\includegraphics{#1.eps}\fi}

\DeclareMathOperator{\cfun}{\theta} 
\DeclareMathOperator{\tfun}{\tau} 

\DeclareMathOperator*{\clco}{\closure\conv}
\DeclareMathOperator*{\compose}{\circ}
\DeclareMathOperator*{\closure}{\text{cl}}
\DeclareMathOperator*{\conv}{\text{conv}}
\DeclareMathOperator{\diff}{\partial}
\DeclareMathOperator{\hadamard}{\odot}
\DeclareMathOperator*{\identity}{\mathbf{id}}
\DeclareMathOperator*{\rank}{\text{rank}}
\DeclareMathOperator*{\radius}{\varrho}
\DeclareMathOperator*{\transdir}{\pi}
\providecommand{\CLARKE}{\text{C}}
\providecommand{\FRECHET}{\text{F}}
\providecommand{\LIMITING}{\text{L}}
\providecommand{\PROXIMAL}{\text{P}}
\providecommand{\FT}{\text{Q}}
\providecommand{\RUBINOV}{\text{S}}
\providecommand{\SINGULAR}{\text{L}_\infty}
\providecommand{\VISCOSITY}{\text{V}}
\providecommand{\cone}[1]{N^{#1}}

\DeclareMathOperator{\couplingcone}{\cone{\RUBINOV}}
\DeclareMathOperator{\decouplingcone}{\cone{\FT}}

\DeclareMathOperator{\proximal}{\subdiff{\PROXIMAL}}
\DeclareMathOperator{\viscosity}{\subdiff{\VISCOSITY}}
\DeclareMathOperator{\frechet}{\subdiff{\FRECHET}}
\DeclareMathOperator{\limiting}{\subdiff{\LIMITING}}
\DeclareMathOperator{\singular}{\subdiff{\SINGULAR}}
\DeclareMathOperator{\clarke}{\subdiff{\CLARKE}}
\DeclareMathOperator{\coupling}{\subdiff{\RUBINOV}}
\DeclareMathOperator{\decoupling}{\subdiff{\FT}}

\DeclareMathOperator{\domain}{\text{dom}}
\DeclareMathOperator{\epi}{\text{epi}}
\DeclareMathOperator{\graph}{\text{graph}}
\DeclareMathOperator{\gradient}{\nabla}
\DeclareMathOperator{\hessian}{\nabla^2}

\DeclareMathOperator{\local}{\text{\rm loc}}
\DeclareMathOperator{\proj}{\text{Pr}}

\DeclareMathOperator*{\sign}{\text{sgn}}

\let\ifdprodangle=\iftrue
\providecommand{\subdiff}[1]{\partial^{#1}}

\providecommand{\tdiff}[1]{\text{{#1}-subdifferential~}}

\providecommand{\ORTHO}[1]{\ensuremath{{^{\perp}_{#1}}}}
\providecommand{\gendiff}[1]{\ensuremath{^{({#1})}}}
\providecommand{\ortho}{\ORTHO{}}
\providecommand{\PARA}[1]{\ensuremath{{^{\parallel}_{#1}}}}
\providecommand{\para}{\PARA{}}
\providecommand{\dprodl}{}
\providecommand{\dprodr}{}
\providecommand{\dprod}[2]{\ifdprodangle\bigl\langle\dprodl\else\bigl(\fi{#1},{#2}\ifdprodangle\bigr\rangle\dprodr\else\bigr)\fi}
\providecommand{\toinf}{\to\infty}
\providecommand{\abs}[1]{\lvert#1\rvert}
\providecommand{\card}[1]{\abs{#1}}
\providecommand{\All}{\text{~for all~}}
\providecommand{\all}{\All}

\providecommand{\feasible}{\mathcal{F}}

\providecommand{\inv}{{^{-1}}}

\providecommand{\tdiff}[1]{\text{{#1}-subdifferential}}

\providecommand{\Lcal}{\mathcal{L}}
\providecommand{\Qcal}{\mathcal{Q}}

\providecommand{\Rset}{\mathbb{R}}

\providecommand{\otherwise}{\text{~otherwise~}}
\providecommand{\st}{\text{~s.t.~}}
\providecommand{\dprod}[2]{\ifdprodangle\bigl\langle\dprodl\else\{\fi{#1},{#2}\ifdprodangle\bigr\rangle\dprodr\else\}\fi}
\providecommand{\halfnum}[1]{\frac{#1}{2}}
\providecommand{\Half}{\halfnum{1}}
\providecommand{\norm}[1]{\parallel{#1}\parallel}

\providecommand{\negspc}{\ensuremath{\!\!}}

\newtheorem{corollary}{Corollary}
\newtheorem{definition}{Definition}
\newtheorem{lemma}{Lemma}
\newtheorem{remark}{Remark}
\newtheorem{proposition}{Proposition}
\newtheorem{theorem}{Theorem}

\begin{document}
\title{Generalized Subdifferentials \\of the Sign Change Counting Function}
\author{Dominique Fortin}
\address{Inria, Domaine de Voluceau, Rocquencourt,
B.P. 105, 78153 Le Chesnay Cedex, France }
\email{Dominique.Fortin@inria.fr}
\author{Ider Tseveendorj}
\address{Laboratoire PRiSM, Universit\'e de Versailles,
45, avenue des \'Etats-Unis, 78035 Versailles Cedex, France}
\email{Ider.Tseveendorj@prism.uvsq.fr}

\begin{abstract}
The counting function on binary values is extended to the signed
case in order to count the number of transitions between contiguous
locations. A generalized subdifferential for this sign change counting
function is given where classical subdifferentials remain intractable.
An attempt to prove global optimality at some point, for the 4-dimensional first non trivial
example, is made by using a sufficient condition specially tailored
among all the cases for this subdifferential. 
\end{abstract}
\maketitle
\section{Introduction}\hylabel{sec:intro}
Two main lines of research are currently challenging in the
non smooth optimization field: the rank function \cite{MR2925620},
for its prominent usage in many
real life applications 
and the extended subdifferentials \cite{MR1834382},
for tailoring new
search techniques for improving local solutions.
The rank function is intimately related with both the singular values
and the eigenvalues in case of symmetric matrices
\cite{MR1704112}; the latter case itself, greatly impacts the
decomposition of a function into a difference of convex functions (DC
decomposition). However, the 0 spectrum remains mixed with either
the strictly positive or the strictly negative eigenvalues and most of
its significance is usually lost during the standard treatment which
linearizes the {\em difficult} part of the DC decomposition.
Motivated by this {\em drawback}, we are aiming at a three way
decomposition of the spectrum instead.
On the other hand, more subdifferentials have been proposed
to be able to prove either necessary and sufficient or sufficient
optimality conditions for some classes of non smooth functions
to be minimized. In spite of this stream, we are motivated to
design rather higher order than $\epsilon$ approximated
subdifferentials.

In this article we address both aspects by introducing the sign change counting
function, a natural extension to the sign counting function at the
heart of the rank function.

In \hyautoref{sec:notations}, we remind the notations for the usual
subdifferentials with their inclusion chain;  
in \hyautoref{sec:signcount} we revisit the sign counting function 
results \cite{MR3035526} by giving them a concise proof along with bounds,
in order to facilitate the study of the sign change counting function.
The \FT-subdifferential is introduced in \hyautoref{sec:gensubdiff}
where its relationship with usual and extended
subdifferentials is proven. A sufficient condition is derived
in a particular case and its effectiveness is shown by proving
global optimality of a 1-dimensional function borrowed from
the litterature in this stream.
Finally, a \FT-subdifferential is given for the sign change counting function 
in \hyautoref{sec:signchange} and an attempt to prove the sufficient
condition above is made on the 4-dimensional first non trivial example.

\section{Notations}\hylabel{sec:notations}
Given an extended real valued function $f:\Rset^n\to\Rset\cup\{\infty\}$,
we use standard notations for  domain, graph, epigraph: 
\begin{align*}
\domain(f)=&\bigl\{x\in\Rset^n\lvert f(x)<\infty\bigr\}\\
\graph(f)=&\bigl\{(x,f(x))\in\Rset^n\times\Rset\lvert x\in\domain(f)\bigr\}\\
\epi(f)=&\bigl\{(x,\alpha)\in\Rset^n\times\Rset\lvert f(x)\leq\alpha\bigr\}\\
\end{align*}
and the orthogonal, ball sets are denoted by
\begin{align*}
X\ortho(x)=&\bigl\{y\in\Rset^n\lvert \dprod{y}{x}=0 \bigr\}\\
B(x,\radius)=&\bigl\{y\in\Rset^n\lvert \norm{y-x}\leq\radius\bigr\}.\\
\end{align*}
It is assumed that $f$ is lower semicontinuous, equivalently $\epi(f)$
is a closed subset of $\Rset^n\times\Rset$. The closure and convex hull operators are 
denoted by $\closure$ and $\conv$ respectively and a function $\Rset\to\Rset$
is implictly extended to a map $\Rset^n\to\Rset^n$ by applying it componentwise.

When $f$ is differentiable its gradient is denoted by $\gradient f(x)$,
otherwise $x^\star$ stands for the subgradient of $f$  at $x$ and various the well known subdifferentials are defined respectively in the following ways
\begin{align*}
x^\star\in\proximal{f(x)}&\iff \\ \exists \sigma\in\Rset_+,& 
\st f(x)\leq f(y)-\dprod{x^\star}{y-x}+\sigma\norm{y-x}^2 \all {y\in B(x,\radius)},\\
x^\star\in\viscosity{f(x)}&\iff \\ \exists g\in C^1(\Rset^n\to&\Rset), \gradient g(x)=x^\star
\st f(x)\leq f(y)-g(y)+g(x) \all {y\in B(x,\radius)},\\
x^\star\in\frechet{f(x)}&\iff{\lim\inf}_{y\to x,y\neq x}\frac{f(y)-f(x)-\dprod{x^\star}{y-x}}{\norm{y-x}}\geq 0,\\
x^\star\in\limiting{f(x)}&\iff \exists x_n,x_n^\star\in\frechet f(x) \st \begin{cases}
\lim_{n\toinf} x_n=x\\
\lim_{n\toinf} f(x_n)=f(x)\\
\lim_{n\toinf} x^{\star}_n=x^\star
\end{cases},\\
x^\star\in\singular{f(x)}&\iff \exists x_n,x_n^\star\in\frechet{f(x)}, \exists t_n\to 0^+ \st
\begin{cases}
\lim_{n\toinf} x_n=x\\
\lim_{n\toinf} f(x_n)=f(x)\\
\lim_{n\toinf} x^{\star}_n=\infty\\
\lim_{n\toinf} t_nx^{\star}_n=x^\star
\end{cases},\\
x^\star\in\clarke{f(x)}&\iff x^\star\in
\clco \bigl(\limiting{f(x)}+\singular{f(x)} \bigr), x\in\domain(f).              
\end{align*}

\begin{remark}
Frechet subdifferential appears as a second order approximation
of a regular differential, since 
$\frac{f(y)-f(x)}{\norm{y-x}}\geq \dprod{x^\star}{\frac{y-x}{\norm{y-x}}}$
in the vicinity of $x$.
\end{remark}
For finite dimensional cases, the following inclusion chain is known 
\begin{align*}
\proximal{f(x)}\subset\viscosity{f(x)}=\frechet{f(x)}\subset\limiting{f(x)}\subset\clarke{f(x)}.
\end{align*}

\section{The sign counting function}\hylabel{sec:signcount}
The function $\cfun$ counting the number of components different from $0$ arises
in many applications 
and its subdifferentials have been studied recently in \cite{MR3035526}; 
in this section, we revisit the proofs in a concise way.
We recall that the sign function (or signum function) of a real number $\alpha$ is defined as follows:
\begin{align*}
\sign(\alpha) =1 \begin{cases}
-1, &\mbox { if  }\alpha < 0,\\
0, &\mbox { if  }\alpha = 0,\\
1, &\mbox { if  }\alpha <0.\\
\end{cases}
\end{align*}
Let active index set and other related sets useful for defining them, be
\begin{align*}
{I\para}(x)=&\bigl\{i\lvert \sign(x_i)=0\bigr\},\\
{I\ortho}(x)=&\bigl\{i\lvert i\not\in{I\para}(x)\bigr\},\\
X_{I\para}(x)=&\bigl\{y\in\Rset^n\lvert \sign(y_i)=0 \all i\in I_{I\para}(x)\bigr\},\\
X_{I\ortho}(x)=&\bigl\{y\in\Rset^n\lvert \sign(y_i)=0 \all i\in I_{I\ortho}(x)\bigr\},\\
\cfun(x)=&\card{I\ortho(x)}=\sum\abs{\sign(x_i)}.
\end{align*}
\begin{theorem}[\cite{MR3035526}]\hylabel{thm:countortho}
\begin{align*}
\proximal \cfun(x)=\frechet \cfun(x)=X_{I\ortho}(x)
\end{align*}
\end{theorem}
\begin{proof}
First observe that $X_{I\ortho}(0)=\Rset^n\supset\frechet \cfun(0)$.
\begin{itemize}
\item[$\frechet \cfun(x)\subset X_{I\ortho}(x)$ : ] 
Any $x$ is a local minimum for $\cfun$,
i.e. there exists $\radius$ such that $\cfun(x)\leq \cfun(y)$ for all $y\in B(x,\radius)$.
We 
consider $x\neq 0$, $\cfun(y)=\cfun(x)\iff y\in X(x)$.
For $y\in X(x)\setminus 0$, there exists $\epsilon>0$ such that
for all $0\leq\tau\leq\epsilon$, $x+\tau y\in B(x,\radius)$.
Therefore both $\cfun(x+\tau y)=\cfun(x)$ and $x^\star\in\frechet \cfun(x)$
imply ${\lim\inf}_{\tau\to 0}\frac{-\tau\dprod{x^\star}{y}}{\norm{y}}\geq 0$.
Suppose $\dprod{x^\star}{y}<0$ and consider $\bar y=-y$, then 
$\bar y\in X(x)$ and $\dprod{x^\star}{\bar y}>0$ a contradiction;
therefore $\dprod{x^\star}{y}=0$.
Now, suppose $x^\star\in X\ortho(x)\setminus X_{I\ortho}(x)$
then there are at least 2 indices $i, j$ such that 
$x^\star_i y_i+x^\star_j y_j=0$ for $y\in X(x)$
but then let $\bar y=y$ except for index $i$ whose value is $\bar y_i=-y_i$;
it leads to $x^\star_i \bar y_i+x^\star_j y_j=-2x^\star_i y_i\neq 0$, a contradiction too.
\item[$X_{I\ortho}(x)\subset\proximal \cfun(x)$ : ] 
By lower semicontinuity, certainly $0\in\proximal \cfun(x)$, so  we may assume
$x^\star\neq 0$.
For $y\in B(x,\xi)\subset B(x,\radius)$ 
\begin{itemize}
\item[$\bullet$] if $y-x\in X(x)$ then 
$\dprod{x^\star}{y-x}=0$ together with $\cfun(y)=\cfun(x)$ 
imply $\cfun(y)+\sigma\norm{y-x}^2=\cfun(x)-\dprod{x^\star}{y-x}+\sigma\norm{y-x}\geq \cfun(x)$
and $x^\star\in\proximal \cfun(x)$,
\item[$\bullet$] otherwise, $y\neq x$ with $x^\star\in X(x)$; let $\underline{y}=\arg\min_{y\in B(x,\xi)} \cfun(y)$
and $\underline{\sigma}=\max\biggl(0,\frac{\dprod{x^\star}{y-x}+\cfun(x)-\cfun(\underline{y})}{\norm{y-x}}\biggr)$
which certainly exists.
Therefore for all $\sigma>\underline{\sigma}$ we get $\cfun(y)-\dprod{x^\star}{y-x}+\sigma\norm{y-x}\geq \cfun(x)$.
\end{itemize}
\end{itemize}
\end{proof}
\begin{theorem}[\cite{MR3035526}]
\begin{align*}
\clarke \cfun(x)=X_{I\ortho}(x)
\end{align*}
\end{theorem}
\begin{proof}
From \hyautoref{thm:countortho}, $X_{I\ortho}(x)\subset\clarke \cfun(x)$.
\begin{itemize}
\item $\limiting \cfun(x)\subset X_{I\ortho}(x)$.
 \\
Since $X_{I\ortho}(0)=\Rset^n$, we may assume $0\neq x^\star\in\limiting \cfun(x)$;
by contradiction, 
\begin{align*}
{\lim\inf}_{y\to x\neq y}\frac{\cfun(y)-\cfun(x)-\dprod{x^\star}{y-x}}{\norm{y-x}}> 0\\
\dprod{x_n^\star}{x_n-x}=0
\end{align*}
since $x^\star\not\in X_{I\ortho}(x)$ by contradictory assumption and $x_n^\star\in\frechet \cfun(x_n)$
implies the latter by previous theorem.
However, for $x\neq 0$ we have seen that for sufficiently small $\radius$
and $x_n\in B(x,\radius)$, we have $\cfun(x_n)-\cfun(x)\leq 0$ which violates first strict inequality.
Overall $\limiting \cfun(x)= X_{I\ortho}(x)$.
\item $X_{I\ortho}(x)$ is convex;  by definition, $x_1,x_2\in X_{I\ortho}(x)$ implies $(\alpha x_1+(1-\alpha) x_2)\in X_{I\ortho}(x)$
and $\cfun(\alpha x_1+(1-\alpha) x_2)\leq \card{I(x)}$ for $0\leq\alpha\leq 1$.
\item $\singular \cfun(x)=\emptyset $.
\\
Since projection of $X_{I\ortho}(x)$ on every $i\in I$ is $\proj_i(X_{I\ortho}(x))=\Rset$ is unbounded,
let $\infty_I\in\Rset^n$ be the point whose components in $I$ are infinite. By abuse of notation,
we may consider $\infty_I\in\limiting \cfun(x)$ hence the emptyness of singular limiting subdifferential.
\end{itemize}
The result follows by convexity of $X_{I\ortho}(x)=\limiting \cfun(x)$.
\end{proof}
Let $e$ denote the all 1s vector in $\Rset^n$ and $\abs{\cdot}$ stands for the componentwise absolute value of a given vector.
\begin{lemma}\hylabel{lem:signminor}
\begin{align*}
\dprod{e}{\frac{\abs{x}}{\norm{x}}}\leq \cfun(x)=\norm{\sign(x)}^2
\end{align*}
\end{lemma}
\begin{proof}
\begin{align*}
x_i^2\leq x_i^2+\sum_{j\neq i} x_j^2\Rightarrow \frac{\abs{x_i}}{\norm{x}}\leq\abs{\sign(x_i)}
\end{align*}
\end{proof}
In \cite{MR2162512,MR2162513} the authors  proved  that 
$f\compose \sigma$ is lower semicontinuous around the matrix $M\in\Rset^{m\times n}$
for  an absolute symmetric function $f$, which is lower semicontinuous around the singular value vector $\sigma(M)\in\Rset^{\min(m,n)}$.
Using this result, one can calculate a subdifferential for the $\rank$ function thanks to its decomposition  $\cfun\compose \sigma$ 
(see \cite{MR3035526}). 
It could be specialized to symmetric matrices $M$ around eigenvalues
under, say decreasing order for the uniqueness of orthogonal matrix
$U$ such that $M=U\lambda(M)U\inv$ (see \cite{MR1704112,MR1687292}).


\section{More subdifferentials}\hylabel{sec:gensubdiff}
By discarding local conditions in viscosity subdifferential, several authors 
(see \cite{MR1834382} and references therein) proposed a natural extension
of both subdifferential and normal cone.
\begin{definition}\hylabel{def:qdiff}
Let $\Lcal$ be a set of real valued and continuous functions 
$l: \Rset^n\to \Rset\cup\pm\infty$
\begin{align}
\coupling f(x)=&\bigl\{ l\in \Lcal\lvert f(x+d)\geq f(x)+l(x+d)-l(x), \all d\in \Rset^n \bigr\},\\
\couplingcone (x;D) =& \bigl\{l\in \Lcal\lvert l(x+d)-l(x)\leq 0, \all d\in D-x \bigr\},
\end{align}
\end{definition}
where $d\in D-x$ stands for $(x+d)\in D$ to stress the condition on the direction at a given point $x$.

In \cite{MR2148133,MR1449790,MR2324960,MR2496428,MR1856798,MR2602910,MR2336771}, for non convex optimization problems global optimality conditions 
are  derived in terms of  $\coupling f(x),\couplingcone (x;D)$.

Recently, the authors \cite{ft2013} extended above subdifferential further by decoupling
the domain space from the directional change rather than composing it along some direction
and obtained  necessary and sufficient conditions for global optimality too.
\begin{definition}[\tdiff{\FT}$\negspc$]\hylabel{def:qdiff}
Let $\Qcal$ be a set of real valued functions 
$q : \Rset^n\times \Rset^p\to \Rset\cup\pm\infty$
continuous in the first argument and let a map $\transdir : \Rset^p\to \Rset^n$
\begin{align}
\decoupling f(x)=&\bigl\{ q\in \Qcal\lvert f(x+\transdir(d))\geq f(x)+q(x,d)-q(x,0), \all d\in \Rset^p \bigr\}\\
\decouplingcone (x;D) =& \bigl\{q\in \Qcal\lvert q(x,d)-q(x,0)\leq 0, \all \transdir(d)\in D-x\bigr\}
\end{align}
\end{definition}
An element $q \in \decoupling f(x)$ is called a $\FT$-subgradient of $f$ at $x$
and $\transdir(d)\in D-x$ stands for $(x+\transdir(d))\in D$ in
$\FT$-normal cone definition. However, in the lack of well accepted
terminology, the image of a $\FT$-subgradient is improperly refered
to as
a $\FT$-subdifferential (as much as in the $\RUBINOV$-case given by Rubinov) until a right naming is widely accepted.

Clearly, for $q_1,q_2\in\decoupling f(x)$ and $\alpha\in[0,1]$ the following inequalities hold
\begin{align*}
\alpha f(x+\transdir(d))&\geq \alpha(f(x)+q_1(x,d)-q_1(x,0))\\
(1-\alpha) f(x+\transdir(d))&\geq (1-\alpha)(f(x)+q_2(x,d)-q_2(x,0))\\
f(x+\transdir(d))&\geq f(x)+(\alpha q_1+(1-\alpha)q_2)(x,d)-(\alpha q_1+(1-\alpha)q_2)(x,0)
\end{align*}
which proves the {\em convexity} of $\decoupling f(x)$.

Observe that for the special case  $\transdir=\identity$, the identity, we get
\begin{align*}
\decoupling f(x)=&\bigl\{ q\in \Qcal\lvert f(x+d)\geq f(x)+q(x,d)-q(x,0), \all d\in \Rset^n \bigr\}
\end{align*}
and moreover $\coupling f(x) \subset \decoupling f(x)$ by choosing simply $l(y)=q(\cdot,y)$.


In order to relate both subdifferentials to Clarke subdifferential, let us recall the standard
and extended directional derivatives relative to the direction $d\in\Rset^n$
\begin{align*}
f'(x;d)=&\lim_{t\to 0^+}\frac{f(x+td)-f(x)}{t}\\
f^\CLARKE(x;d)=&\limsup_{y\to x,t\to 0^+}\frac{f(y+td)-f(y)}{t}\\
f^\RUBINOV(x;d)=&\sup\bigl\{l(x+d)-l(x)\lvert l\in\coupling f(x)\bigr\}\\
f^\FT(x;d)=&\sup\bigl\{q(x,d)-q(x,0)\lvert q\in\decoupling f(x)\bigr\}
\end{align*}
so that
\begin{align*}
\xi\clarke f(x) \iff \dprod{\xi}{x+d}-\dprod{\xi}{x}\leq f^\CLARKE(x;d),\quad\forall d\in\Rset^n\\ 
l\in\coupling f(x) \iff l(x+d)-l(x)\leq f^\RUBINOV(x;d),\quad\forall d\in\Rset^n\\
q\in\decoupling f(x) \iff q(x,d)-q(x,0)\leq f^\FT(x;\transdir(d)),\quad\forall d\in\Rset^p
\end{align*}
\begin{proposition}\hylabel{prop:couplingClarke}
Let suppose $f^\RUBINOV(x;d)$ be positive subhomogeneous in $d$, namely $f^\RUBINOV(x;td)\leq t f^\RUBINOV(x;d)$ for all $t>0$,
then
\begin{align*}
l\in\coupling f(x) &\Rightarrow f^\CLARKE(x;d)\leq l'(x;d)\\
\clarke f(x)&\subset\coupling f(x)
\end{align*}
\end{proposition}
\begin{proof}
Let $g(x)=l(x+td)-l(x)$,
we apply the Ekeland principle
for $g$, not identically 0 and  bounded by below, such that $g(x)\leq \inf g(x)+\epsilon$ for $\epsilon>0$,
to get
\begin{align*}
\exists v,  \lim_{t\to 0^+}\frac{g(x)}{t}&\geq \lim_{v\to x,t\to 0^+}\frac{g(v)}{t}\\
&\geq \limsup_{v\to x,t\to 0^+}\frac{f(v+td)-f(v)}{t}=f^\CLARKE(x;d)
\end{align*}
where we select the direction $d$ such that $x=v+td$ so that $f(v+td)-f(v)\leq l(v+td)-l(v)$ since $l\in\coupling f(x)$.
Then, by assumption  $f^\RUBINOV(x;d)\geq \lim_{t\to 0^+}\frac{g(x)}{t}$ for all $l\in\coupling f(x)$, hence we conclude 
$f^\RUBINOV(x;d)\geq f^\CLARKE(x;d)$ and the expected inclusion is proven.
\end{proof}

\begin{remark}
$\quad $
\begin{enumerate}
\item If $l'(x;d)$ is weakened to $l^\CLARKE(x;d)$ instead, then the result holds  without the Ekeland principle
and we retrieve the sublinear property that may fail for $l'(x;d)$ provided $f^\RUBINOV(x;d)\geq l^\CLARKE(x;d)$.
\item For locally Lipschitz function $f$, in the Frechet setting we have more directly
\begin{align*}
\clarke f(x)=\conv\bigl\{\lim\gradient f(x_k)\lvert x_k\to x, x_k\not\in \Omega_f, x_k\not\in S\bigr\}
\end{align*}
Let $l(x)=\sigma\norm{x-x_k}^2+\min\bigl\{ \dprod{\gradient f(x_k)}{x-x_k}\bigr\}$ 
where the minimum is over all linear functions $\dprod{\gradient f(x_k)}{x-x_k}$
with $\gradient f(x_k)\in \clarke f(x)$
then
$f(x)\leq f(x_k)+l(x)-l(x_k)$ for some $\sigma>0$
and $\clarke f(x)\subset\coupling f(x)$. 
\end{enumerate}
\end{remark}

\begin{proposition}\hylabel{prop:decouplingClarke}
Let us assume that $q(x,\cdot)\in\Qcal$ be lower semi-continuous and $f^\FT(x;d)$ be positive subhomogeneous, both in the second argument,
then
\begin{align*}
q\in\decoupling f(x) &\Rightarrow f^\CLARKE(x;\transdir(d))\leq q'_{x,0}(x;d)\\
\clarke f(x)&\subset\decoupling f(x)
\end{align*}
\end{proposition}
where $q'_{x,0}(x;d)= \lim_{t\to 0^+}\frac{q(x,td)-q(x,0)}{t}$ is a directional derivative at $0$, with respect to the second argument.
\begin{proof}
Let $h(x)=q(x,td)-q(x,0)$; we apply the Ekeland principle 
for $h$, not identically 0 and  bounded by below, such that $h(x)\leq \inf h(x)+\epsilon$ for $\epsilon>0$,
to get
\begin{align*}
\exists y,  \lim_{t\to 0^+}\frac{h(x)}{t}&\geq \lim_{y\to x,t\to 0^+}\frac{h(y)}{t}\\
&\geq \limsup_{y\to x,t\to 0^+}\frac{f(y+t\transdir(d))-f(y)}{t}=f^\CLARKE(x;\transdir(d))
\end{align*}
Using the assumption $f^\FT(x;td)\leq tf^\FT(x;d)$, we conclude $f^\FT(x;d)\geq f^\CLARKE(x;\transdir(d))$ and the expected inclusion is proven
under mild assumptions on $\transdir$.
\end{proof}

\begin{remark}
Like in the $\coupling f(x)$ case, the Ekeland principle could be skipped, provided
the subhomogeneity is replaced by a decoupled Clarke's derivative in the first argument at $x$ and 
a directional derivative in the second argument at $0$\\
$f^\FT(x;d)\geq \lim_{y\to x,t\to 0^+}\frac{q(y,td)-q(y,0)}{t}\simeq q^{\CLARKE~'}_{x,0}(x;d)$.
\end{remark}

There is no obvious composition rule for both subdifferentials unless the functions
in $\partial f(x)$ have a representation in $\Rset^n$.

\section{The sign change counting function}\hylabel{sec:signchange}
The number of sign transitions along the components may be seen as a first order
derivative of the sign counting function; its active index set is defined
accordingly
\begin{align*}
J(x)=&\bigl\{i\lvert \sign(x_i)=\sign(x_{i+1})\bigr\}\\
{J\para}(x)=&\bigl\{i,i+1\lvert i\in J(x))\bigr\},\quad {J\ortho}(x)=\bigl\{i\lvert i\not\in{J\para}(x)\bigr\}\\
X_{J\para}(x)=&\bigl\{y\in\Rset^n\lvert\sign(y_i)= \sign(x_i) \all i\in {J\para}(x)\bigr\}\\
\tfun(x)=&\card{J\ortho(x)}
\end{align*}
where $i+1$ is possibly taken modulo $n$ for circular counting (as required in some 
real life applications).
For instance, at $x=(-24 ,-30, 19,14, 0)^{\top}\in \Rset^5$ the sign change counting function output is $\tfun(-1,-1,1,1,0)=3$.

\begin{lemma}\hylabel{lem:changecomp}
Let assume $x_{n+1}=x_1, x_{n}$ according with circular case or not, and define $l_i(x;k)=(\sign(x_i)+\sign(x_{i+1})+k\sign(x_ix_{i+1}))(\sign(x_ix_{i+1})-1)$ for the parameter $k$
then $l_i(x;k)=0$ iff $i\in J(x)$ provided $k\neq 0$.
Moreover, $\abs{l_i(x;1)}$ measures the gap between the 2 consecutive components.
\end{lemma}
\begin{proof}
Observe that $l_i(x;k)=(k+2\sign(x_i))(1-1)=0$ whenever $x_{i+1}=x_i$ to deal with  non circular case
and no sign change,
then by enumeration of remaining cases,
{\tiny
\begin{align*}
l_i(x;k)=&\begin{cases}
(-1+0+0k)(0-1)=1\\
(-1+1-k)(-1-1)=-2k\\
(0-1-0k)(0-1)=1\\
(0+1+0k)(0-1)=-1\\
(1-1-k)(-1-1)=2k\\
(1+0+0k)(0-1)=-1
\end{cases}
\end{align*}
}
\end{proof}
As a direct consequence, since $l_i(0;k)=0$ for all $i$
\begin{corollary}\hylabel{cor:changecomp}
Let $l:\Rset^n\to \Rset^n$ be the map defined by 
$$l(x;k)=\begin{bmatrix}l_1(x;k)&\ldots&l_n(x;k)\end{bmatrix}^{\top}$$
then $l(0;k)=0$ and 
\begin{align*}
\cfun(l(x;0<\abs{k}<1/2))\leq \tfun(x)=\cfun(l(x;\pm 1/2))=\norm{l(x;\pm 1/2)}^2\leq \cfun(l(x;1/2<\abs{k}))
\end{align*}
\end{corollary}
since the absolute value emphasizes a sign change for $k=\pm 1/2$, while
it is respectively minorized and majorized otherwise.
Notice that the parameter is only involved in $\pm 1$ sign transitions
and $k=\pm 1$ amplifies the sign change count for such transitions by
a factor 2,  therfeore $\abs{k}=1$ play the role of an upper bound somehow.
A smoothed minorant version directly follows the smoothed version of the sign
counting function
\begin{align*}
\begin{cases}
\tfun(x)\geq\tfun_\epsilon(x)=\cfun_\epsilon(l(x;\pm\Half))\\
\lim_{\epsilon\to 0}\tfun_\epsilon(x)=\tfun(x).
\end{cases}
\end{align*}

For $n=2$ the function $\tilde l_1: \Rset^2\to\Rset$ defined by
$$\tilde l_1(x;\pm 1/2)=(x_1+x_2\pm x_1x_2/2)(x_1x_2-1)$$ one can calculate the hessian
{\small
\begin{align*}
\hessian \tilde l_1(x;\pm 1/2)=\begin{bmatrix}
x_2(2\pm x_2)&2(\mp 1/4+x_1+x_2\pm x_1 x_2)\\
2(\mp 1/4+x_1+x_2\pm x_1 x_2)&x_1(2\pm x_1) \end{bmatrix}
\end{align*}
}
with eigenvalues at $sgn(x)$ points, with opposite signs (which means $\tilde l_1$ neither convex nor concave)
\begin{align*}
2\Lambda\begin{bmatrix}-1&-1\end{bmatrix}^t=\mp 2\pm 5\\
2\Lambda\begin{bmatrix} 0& 0\end{bmatrix}^t=\pm 1\\
2\Lambda\begin{bmatrix}+1&+1\end{bmatrix}^t=\pm 6\mp 11\\ 
2\Lambda\begin{bmatrix}-1&0\end{bmatrix}^t=2\Lambda\begin{bmatrix}0&-1\end{bmatrix}^t=\mp 1\pm \sqrt{26}\\
2\Lambda\begin{bmatrix}-1&+1\end{bmatrix}^t=2\Lambda\begin{bmatrix}+1&-1\end{bmatrix}^t=\pm 2\pm \sqrt{41}\\
2\Lambda\begin{bmatrix} 0&+1\end{bmatrix}^t=2\Lambda\begin{bmatrix}+1& 0\end{bmatrix}^t=\pm 3\pm 3\sqrt{2}
\end{align*}
so that for all $i$, $\tilde l_i(x;\pm 1/2)$ is twice differentiable at $\sign(x)$.

Under a slight abuse of notation, where  $\norm{l(x;\pm 1/2)}^2$ stands for $l(x)$,
and  under the special transform $\transdir=\identity$ we  prove our main result of the section.
\begin{theorem}\hylabel{thm:changecountdiff}
\begin{align*}
\norm{l(x;\pm 1/2)}^2&\in\coupling \tfun(x)\\
q(x,y;k_x,k_y)-q(x,0;k_x,k_y)
&=\norm{l(x+y;0<\abs{k_y}\leq 1/2)}^2-\norm{l(x;1/2\leq\abs{k_x})}^2\\
q(x,y;k_x,k_y)&\in\decoupling \tfun(x)
\end{align*}
\end{theorem}
\begin{proof}
$\tfun(x)=\tfun(y)-\norm{l(y,\pm 1/2)}^2+\norm{l(x;\pm 1/2)}^2$ leads to the
  coupled case $\coupling \tfun(x)$;
however, it does not help much as a direction for minimizing
the number of sign changes, as a neutral subgradient.
The decoupled case follows 
by applying \hyautoref{cor:changecomp},
$ \tfun(y)-\tfun(x)\geq 
\norm{l(y;0<\abs{k_y}\leq 1/2)}^2
-\norm{l(x;1/2\leq\abs{k_x})}^2$.

Notice that 
\begin{align*}
&\hat q(x,0;0<\abs{k_y}\leq 1/2\leq\abs{k_x})=\norm{l(x;0<\abs{k_y}\leq 1/2)}^2-\norm{l(x;1/2\leq\abs{k_x})}^2\\
&\quad= (k_y-k_x)\sum
(x_ix_{i+1}-1)^2\bigl((k_y+k_x)x_i^2x_{i+1}^2+2x_ix_{i+1}(x_i+x_{i+1})\bigr)
<0\end{align*}
where for sake of brevity, we write $x_i$ instead
of $\sign(x_i)$; clearly,
each summand accounts for $0$ if either one component is $0$, or both 
components have same sign, and for $4 (k_y^2-k_x^2)<0$ otherwise.
Therefore 
\begin{align*}
\tfun(y)-\tfun(x)\geq& \hat q(x,y-x;0<\abs{k_y}\leq 1/2\leq\abs{k_x})\\
\geq& \hat q(x,y-x;0<\abs{k_y}\leq 1/2\leq\abs{k_x})
+\hat q(x,0;0<\abs{k_y}\leq 1/2\leq\abs{k_x})
\end{align*}
showing that $q$ is implicitly known.
However, it improves over $\coupling \tfun(x)$ with the parameters $k_y, k_x$ varying.

\end{proof}

\begin{remark}
In a domain $D$ in the vicinity of $x$, parameters $k_x$ and $k_y$
may vary such that $q(x,d;k_x,k_y)\in\decouplingcone(x;D)$.
It would be interesting to check whether a similar situation
arises for $\couplingcone(x;D)$ for a single parameter $k$ varying.
\end{remark}

\begin{remark}
Let $I$ (resp. $Z$) be the identity (resp. circular shift) matrix conformably to the vector size
$x$ and $\hadamard$ the componentwise (Hadamard) product;
the $i-$th power composition of $\hadamard$ is denoted by a self
explanatory exponent $()^{\hadamard i}$. In this algebraic writing,
the relationship $\norm{x\hadamard y}^2=\dprod{x^{\hadamard
    2}}{y^{\hadamard 2}}$ gives
\begin{align*}
\norm{l(x;k)}^2=
\dprod{\bigl((I+Z)x+k(x\hadamard Zx)\bigr)^{\hadamard  2}}{\bigl((x\hadamard Zx)-e\bigr)^{\hadamard 2}}
\end{align*}
with the all $1$s vector $e$.
\end{remark}

\begin{lemma}\hylabel{lem:changecountfrechet}
For all $x\not\in\arg\local\min t$ and $x\neq 0$ 
\begin{align*}
\begin{cases}
\frechet \tfun(x)\subset X\ortho(x), x\in\arg\local\max t\\
\frechet \tfun(x)=0, \otherwise
\end{cases}
\end{align*}
\end{lemma}
\begin{proof}
Frechet differential is computed, in the same way as for the sign function
for simple cases
\begin{itemize}
\item[$\bullet$] no $0$ component in the vicinity of $x$.
For all $x+\tau y\in B(x,\radius)$, $\tfun(x+\tau y)=\tfun(x)$,  hence 
$\dprod{x^\star}{y}\leq 0$ for all $x^\star\in\frechet \tfun(x)$.
Then, suppose $\dprod{x^\star}{y}< 0$ and consider $\bar y=-y$ to arrive at the contradiction
$\dprod{x^\star}{\bar y}> 0$ for $x^\star\in\frechet \tfun(x)$.
Finally, suppose $x_j^\star\neq 0$ for some $j$ and
$\sum x_i^\star y_i=0$, since there is no sign change in $B(x,\radius)$
taking $\bar y=y$ except for taking the opposite of $y_j\neq 0$, we get 
$\sum x_i^\star \bar y_i=-2 x_j^\star y_j\neq 0$.
It yields $\frechet \tfun(x)=0$ and we consider $x$ to be stationary.
\item[$\bullet$] $X\ortho(0)=\Rset^n\supset\frechet \tfun(0)$.
\item[$\bullet$] $x\neq 0$ is a local maximum.
For all $x+\tau y\in B(x,\radius)$, $\tfun(x+\tau y)\leq \tfun(x)$,  
and we get $\dprod{x^\star}{y}=0$, provided $x^\star\neq 0$ to derive
both previous contradictions. Therefore $\frechet \tfun(x)\subset X\ortho(x)$.
\end{itemize}
\end{proof}
If $0\neq x\in\arg\local\min t$ then  $\frechet \tfun(x)$ raises combinatorial
enumeration depending on the number of $\pm 1$ signs; this issue does not arise
for $\decoupling \tfun(x)$.

As an example, consider the $\tdiff{\FT} \tfun(x)$ at
a given transition point $x$ and direction $d$
\begin{align*}
q'(x;d)=-310+692 k-888 k^2
\begin{cases}
x=\begin{bmatrix}
-1& 1& 1& 0& -1& 0& 0\end{bmatrix}^t\\
d=\begin{bmatrix}
0& 74& 75& 0& -40& -50& 0\end{bmatrix}^t
\end{cases}
\end{align*}
for the varying parameter $0<k\leq 1/2$,  since the parameter $1/2\leq k_x$ does not play any role
in this instance
\hyautoref{fig:qder}.
The quadratic term justifies the naming of subdifferential.

 \begin{figure}[h]\hylabel{fig:sc7}
\begin{center}
\scalebox{0.2}{\includegraphix{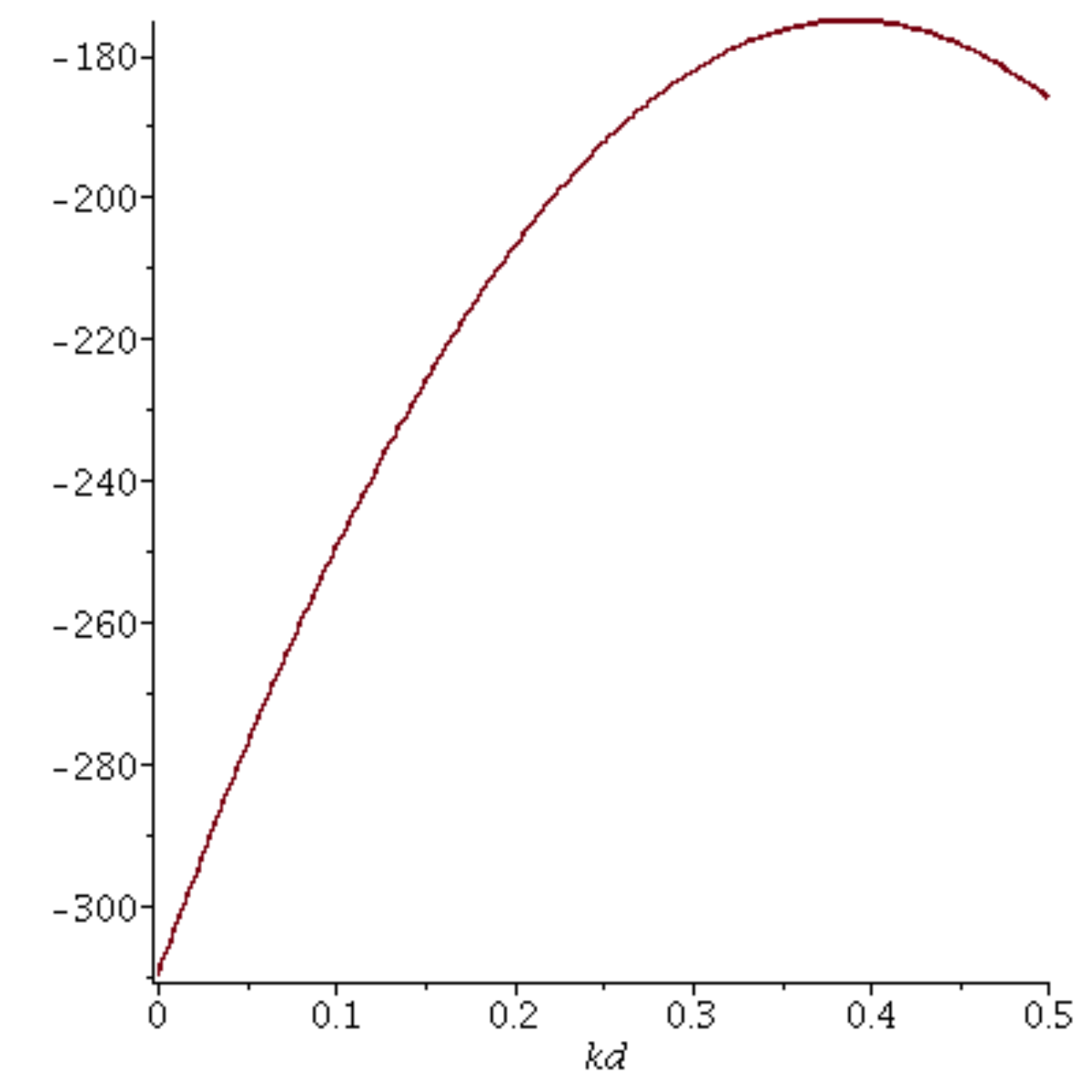}}
\caption{example of $\decoupling \tfun(x)$ for varying parameter $0<k\leq 1/2$}\hylabel{fig:qder}
\end{center}
\end{figure}

\subsection{Example of the continuous multipliers}\hylabel{sec:example1}

Let us remind, to be self contained, the sufficient condition for global minimizer
established by the authors under the identity map $\transdir=\identity : \Rset^n\to \Rset^n$
\begin{proposition}[\cite{ft2013}]\hylabel{prop:CS1}
Let $Q$ be a set of real valued functions  such that $-q\in Q$ for
$q\in Q$;
if there are continuous multipliers $\lambda_i(x)\geq 0$ for all $x \in \feasible$ and
a feasible $z \in \feasible$ fulfilling 
\begin{align}
&0\in\decoupling f(z)+\decoupling\biggl(\sum_{i\in M} \lambda_i g_i\biggr)(z),\hylabel{eq:lag1}\\ 
&\biggl(\sum_{i\in M} \lambda_i g_i\biggr)(z)=0\hylabel{eq:compl}
\end{align}
then $z$ is a global minimizer for $\min f(x)$ under $g_i(x)\leq 0$.
\end{proposition}
We considered in \cite{ft2013} the following example from \cite{MR1449790}.
\begin{align*}
\min &~ x^2\cos(2x)\\
\st&\begin{cases}
-x-2\pi \leq 0\\
 x-2\pi \leq 0
\end{cases}
\end{align*}
Choose multipliers 
$\lambda_1(x)=\abs{x-c_1}\exp{\abs{x-c_1}}$, $\lambda_2(x)=0$
and
$$q_f(x,y;\sigma)=(x+y)^2\cos(2(x+y))-y^2\cos(2y)-\sigma \abs{y}.$$

We get for $c_1=-4.8$, $\sigma=1$ the system of
dependent constraints 
\begin{align*}&\begin{cases}
-\abs{x+4.8}\exp(\abs{x+4.8})(x+2\pi)+x^2\cos(2x)+22.687\geq 0\\
-\abs{x+4.8}\exp(\abs{x+4.8})(x+2\pi)-x^2\cos(2x)-22.687+\sigma\abs{x+4.8}\geq 0\\
2x^2\cos(2x)+45.374-\sigma\abs{x+4.8}\geq 0
\end{cases}\end{align*}
which is fulfilled on the interval $[-2\pi,0]$ (see \hyautoref{fig:fbcond3_4_8}), hence the global
minimizer at $z=-4.8$.
\begin{figure}[ht]
  \begin{center}
\scalebox{0.2}{\includegraphix{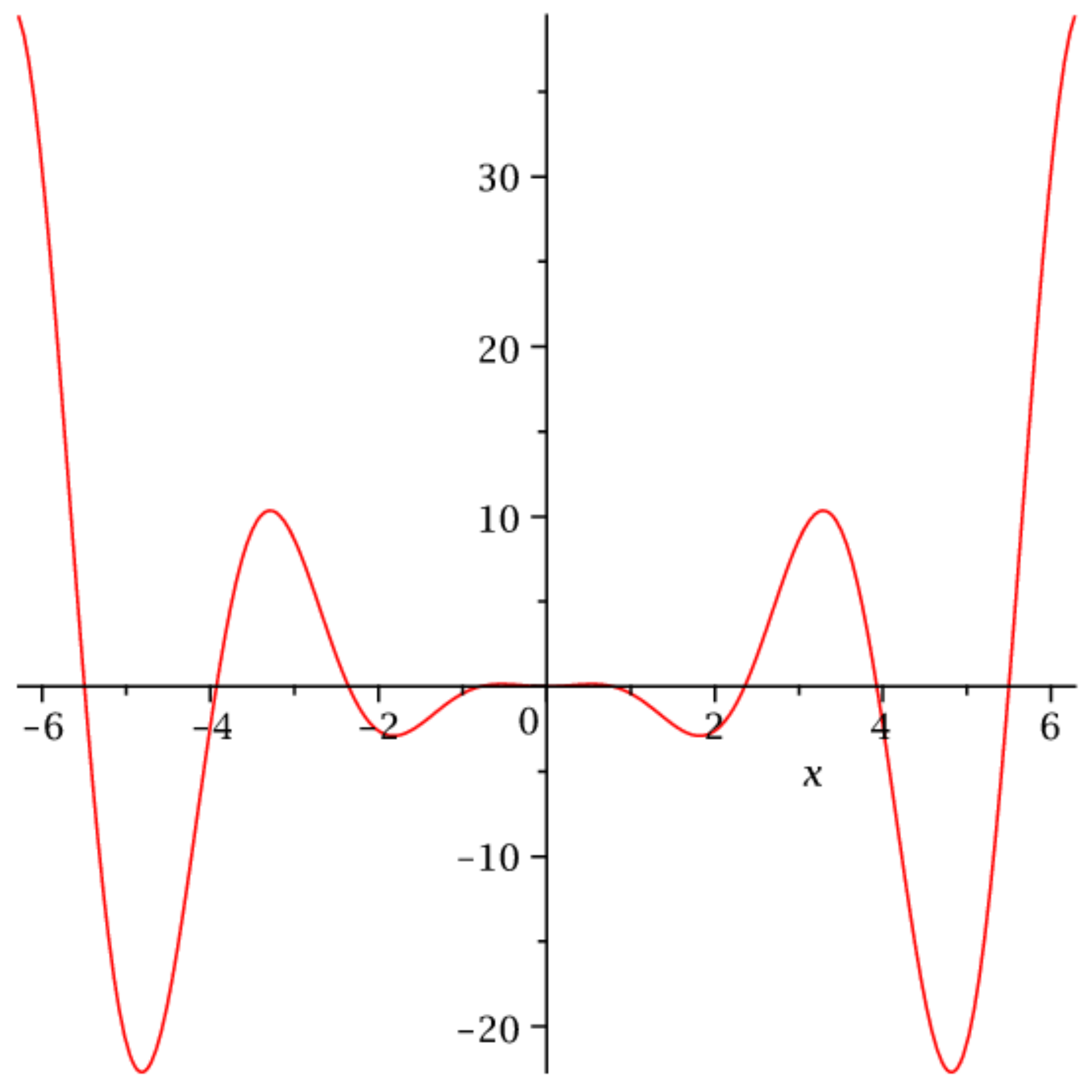}}\quad\scalebox{0.2}{\includegraphix{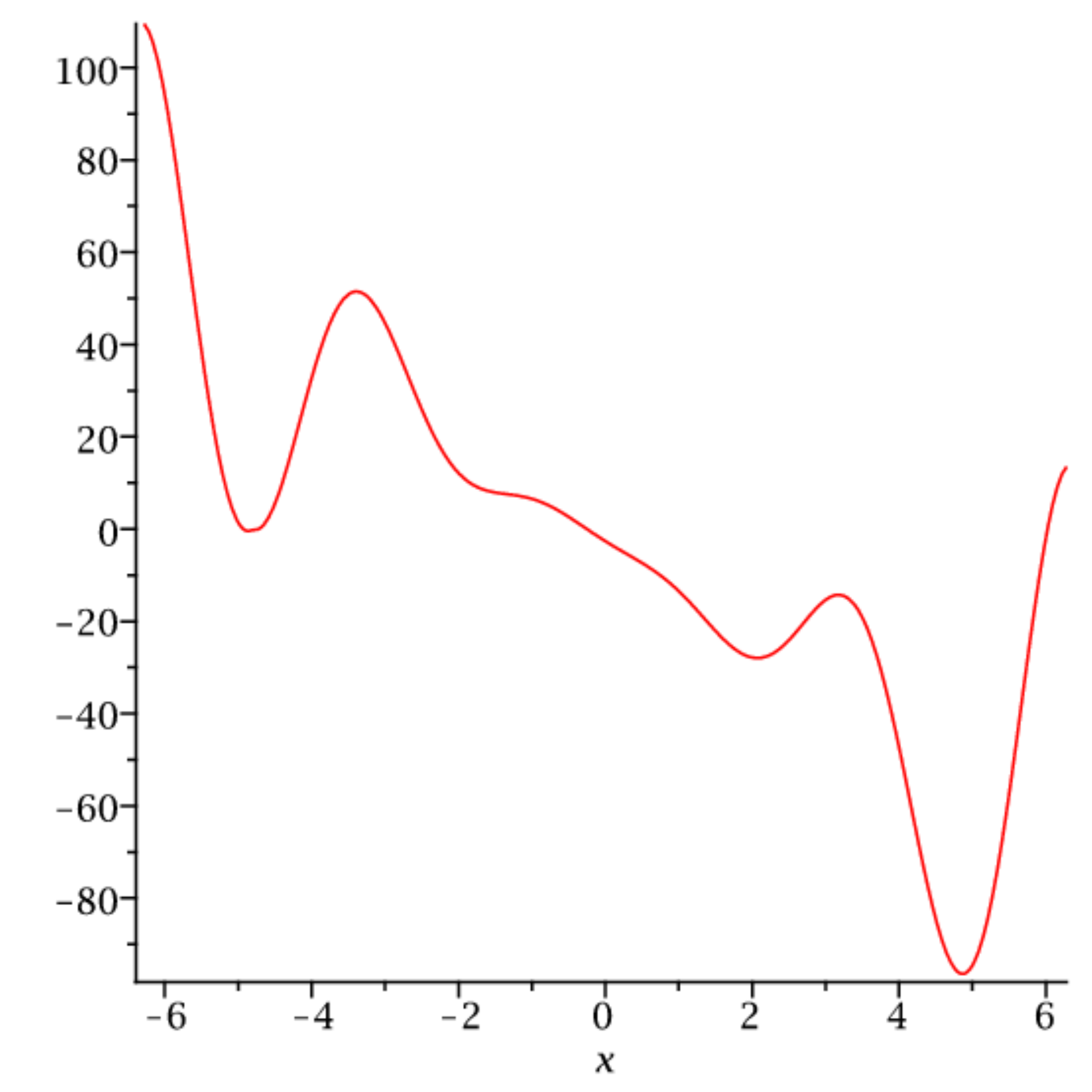}}
\end{center}
\caption{global minimizer at $z=-4.8$: on the left $f$ ; on the right $\lambda_1$
  multiplier with $\sigma=1$}
 \hylabel{fig:fbcond3_4_8}
\end{figure}

\subsection{Example showing some optimality condition checking difficulties}
Splitting equality constraints $g_i(x)=x_i(1-x_i^2)=0$ which capture sign
values $\{0,\pm 1\}$, into $g_i(x)\leq 0$ and
$-g_i(x)\leq 0$, then  
sufficient optimality condition \hyautoref{prop:CS1}
for $z$ to be a
global minimizer of $t$ simplifies by merging both non negative
continuous multipliers into a single one without the non negative condition
\begin{align}
&0\in\decoupling t(z)+\decoupling\biggl(\sum_{i\in M} \lambda_i
  g_i\biggr)(z),\hylabel{eq:lag1}
\end{align}

\begin{align}
\norm{l(z;1/2)}^2-\norm{l(d;k)}^2&\in\decoupling t(z)\\
\dprod{\gradient(\lambda_i g_i)(z)}{d}&=\dprod{\lambda\hadamard(e-3 z^{\hadamard  2})}{d}
\in\decoupling\biggl(\sum_{i\in  M} \lambda_i g_i\biggr)(z)\hylabel{eq:failure}
\end{align}
since for all $i,j$ we have $z_j(1-z_j)^2\diff \lambda_i/ \partial x_i=0$ at $z$.
\subsubsection{trivial examples upto dimension 3}\hylabel{sec:sc3}
For the, obviously optimal, solution,
$z=[-1,1]$, it leads, under the circular case, to the equation: 
\begin{align*}
\lambda_1 d_1(1-3z_1^2)+\lambda_2
d_2(1-3z_2^2)+2&-\bigl((d_1+d_2+k_{12}d_1d_2)^2\\
&\quad+(d_1+d_2+k_{21}d_1d_2)^2\bigr)(d_1d_2-1)^2=0
\end{align*}
w.l.o.g. we may select $k_{21}=k_{12}$ and polar coordinates $d_1=\cos\phi_1$, $d_2=\sin\phi_1$.
There exists $\lambda_1=0$ and $k_{12}=0$ such that
\begin{align*}
4\lambda_2\sin\phi_1=&4-(\cos\phi_1+\sin\phi_1)^2(\sin(2\phi_1)-2)^2
\end{align*}
Symmetry associated with circular case allows us to reduce $d_2$ to non negative values, equivalently
$\phi_1\in[0,\pi]$, then the sufficient condition is fulfilled as could be
checked on \hyautoref{fig:sc2}.
\begin{figure}[h]
\begin{center}
\scalebox{0.4}{\includegraphix{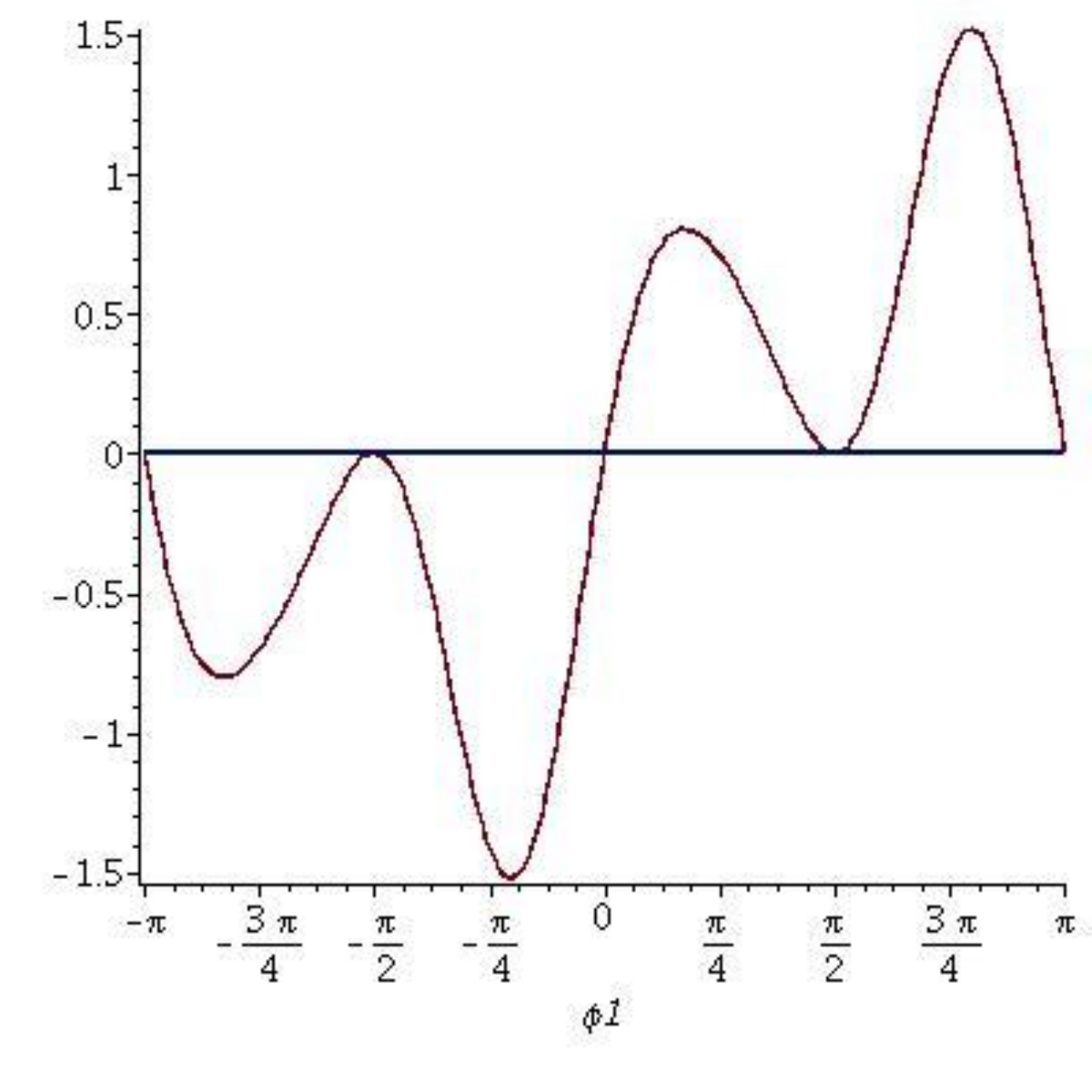}}
\caption{sufficient condition at $z=[-1,1]$ 
  $\lambda_2(\phi_1)$ under  $\lambda_1=0$ }\hylabel{fig:sc2}
\end{center}
\end{figure}

For $z=[1,-1,1]$, with $2$ sign changes, it leads to the equation: 
\begin{align*}
0=&\lambda_1 d_1(1-3z_1^2)+\lambda_2 d_2(1-3z_2^2)+\lambda_3 d_3(1-3z_3^2)+2
-(d_1+d_2+k_{12}d_1d_2)^2(d_1d_2-1)^2\\
&\quad-(d_2+d_3+k_{23}d_2d_3)^2(d_2d_3-1)^2
-(d_3+d_1+k_{31}d_3d_1)^2(d_3d_1-1)^2
\end{align*}
W.l.o.g. let us choose a spherical parametrization $d_1=\cos\phi_1$,
$d_2=\cos\phi_2\sin\phi_1$, and $d_3=\sin\phi_2\sin\phi_1$;
for symmetry reason once more, we may assume $d_3\geq 0$ so that
$\phi_1,\phi_2$ shrink to the open interval $(0,\pi)$.
Select $\lambda_1(x)=\lambda_2(x)=0$ and
$k_{12}=k_{23}=k_{31}=0$, then
\begin{align*}
8\lambda_3\sin\phi_1\sin\phi_2=&8
-(\cos\phi_1+\cos\phi_2\sin\phi_1)^2(\cos\phi_2\sin(2\phi_1)-2)^2
\\&\quad
-\sin^2\phi_1(\cos\phi_2+\sin\phi_2)^2(\sin^2\phi_1\sin(2\phi_2)-2)^2
\\&\quad
-(\cos\phi_1+\sin\phi_1\sin\phi_2)^2(\sin\phi_2\sin(2\phi_1)-2)^2
\end{align*}
\begin{figure}[h]\hylabel{fig:sc3}
\vspace{-.9cm}
\begin{center}
\hspace{-1.0cm}\scalebox{0.4}{\includegraphix{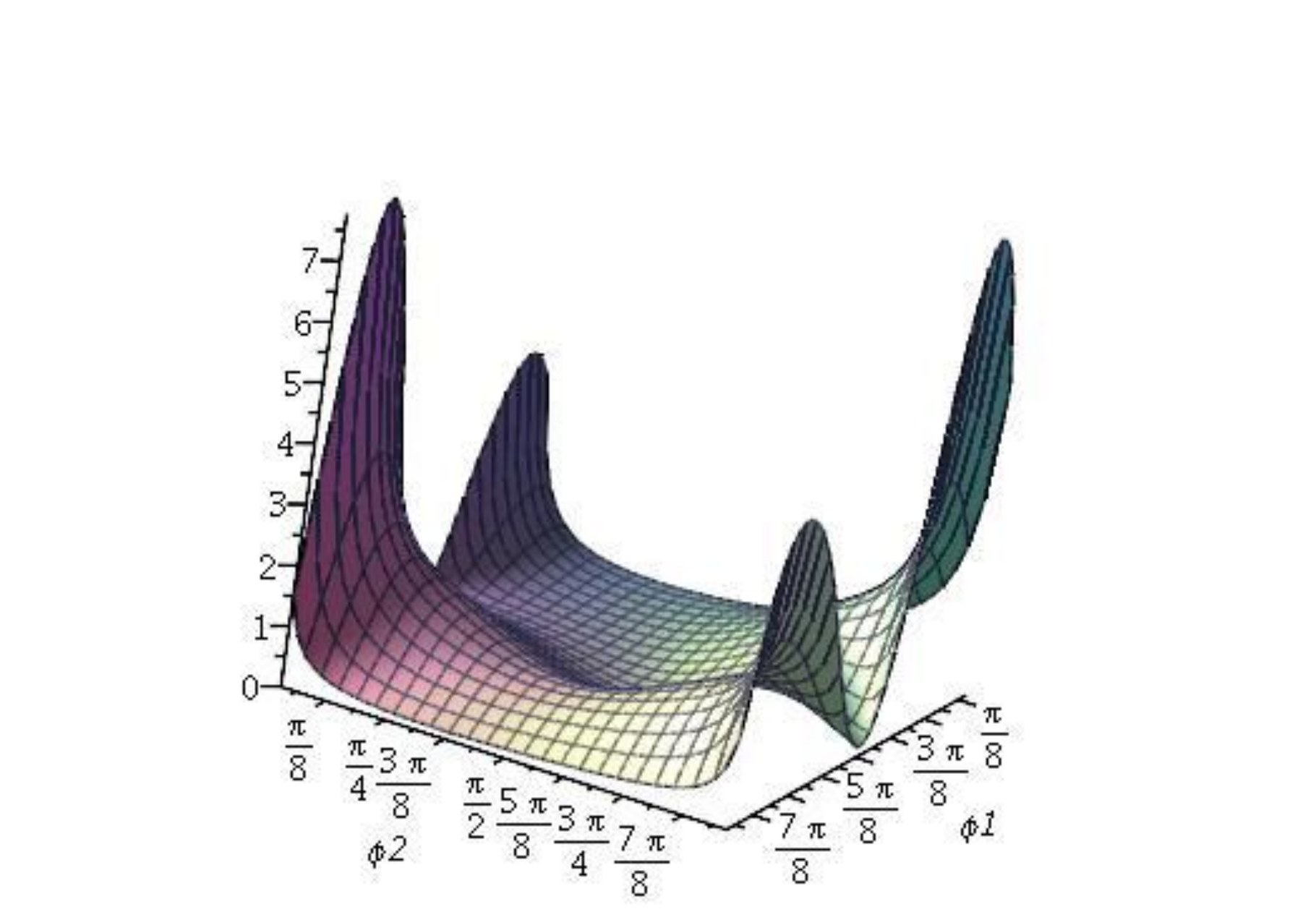}}\hspace{-1.0cm}\scalebox{0.4}{\includegraphix{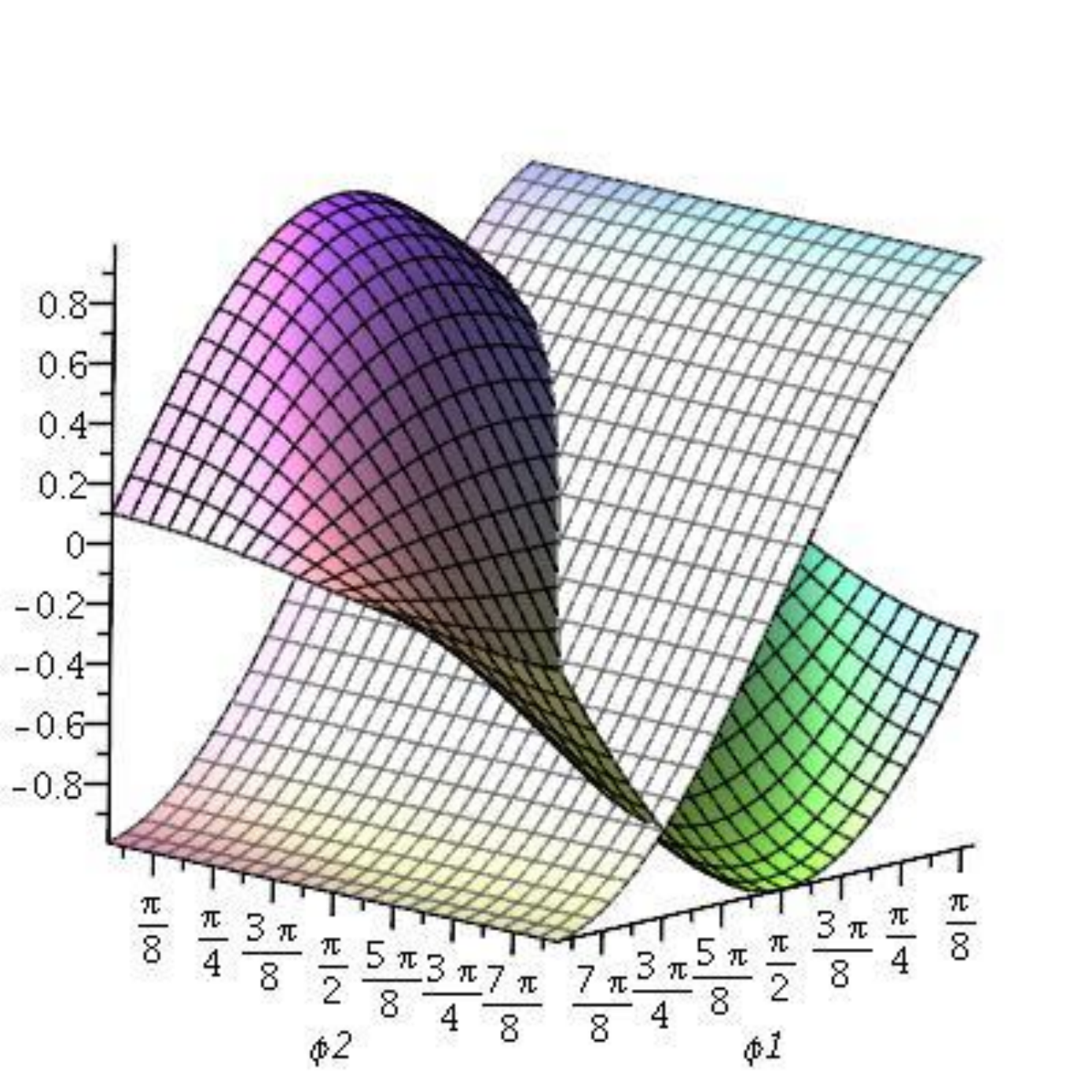}}
\caption{sufficient condition at $z=[1,-1,1]$ 
  $\lambda_3(\phi_1,\phi2)$ under
  $\lambda_1=\lambda_2=0$ (left)  and directions $d_1$, $d_2$ (right)}
\end{center}
\end{figure}

\subsubsection{first non trivial example}\hylabel{sec:sc4}
A careful look at \hyautoref{eq:failure} reveals that, on the contrary
to 1-dimensional example \hyautoref{sec:example1}
that no explicit continuous multipliers are able to fulfill the
sufficient condition since,
on the one hand $\norm{l(z;1/2)}^2-\norm{l(d;k)}^2=0$ whenever
$\abs{k}=\Half$ so that at least one component in left hand side
has to be not $0$ or else every point will be optimal, a contradiction,
and, on the other hand, the left hand side leads to a system
of $3^4-1$ equations in $4$ parameters (whatever there relationship
with the multiplier functions).
Until the sufficient condition is extended to a different case than
identity mapping, the hope remains to find a stronger condition
that would be fulfilled at optimal points without constructing
the multipliers (implicitly at most).

\begin{figure}[ht]
\begin{center}
\ifpdf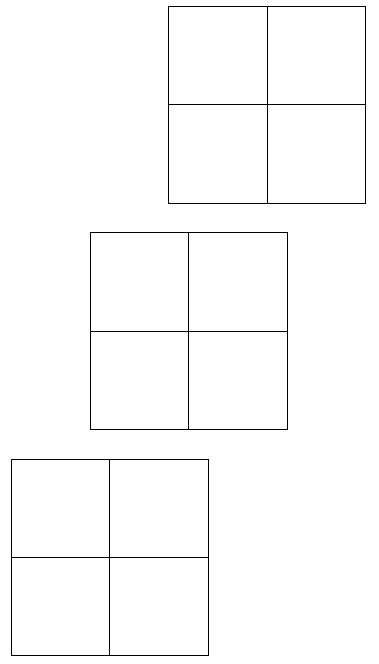\else\input{sc4.pstex_t}\fi
\caption{sign changes at $z_1=-1$}\hylabel{fig:sc4}
\end{center}
\end{figure}
\begin{figure}[ht]
\begin{center}
\ifpdf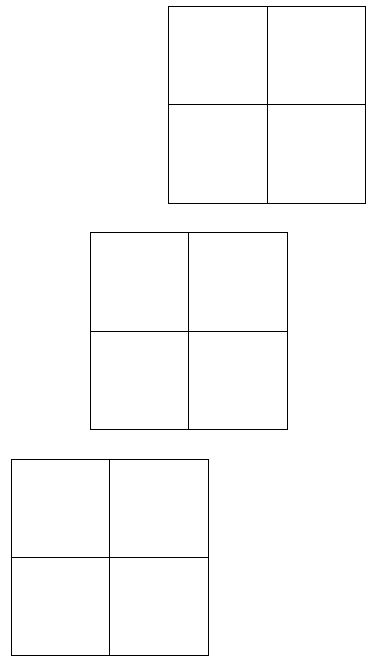\else\input{sc4_0.pstex_t}\fi
\caption{sign changes at  $z_1=0$}\hylabel{fig:sc4_0}
\end{center}
\end{figure}
\begin{figure}[ht]
\begin{center}
\ifpdf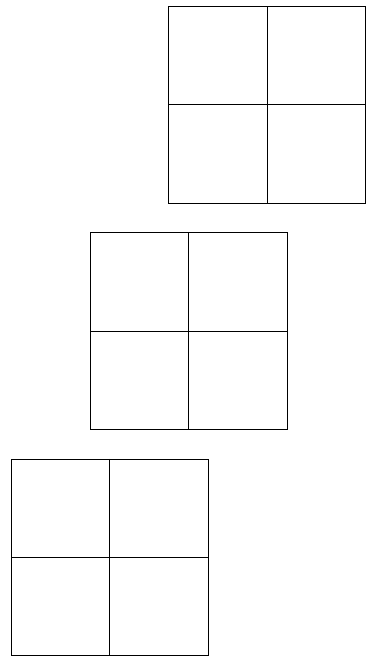\else\input{sc4_1.pstex_t}\fi
\caption{sign changes at  $z_1=1$}\hylabel{fig:sc4_1}
\end{center}
\end{figure}

By enumeration, we can see from \hyautoref{fig:sc4},
\hyautoref{fig:sc4_0},
and \hyautoref{fig:sc4_1} that not optimal points,  $t(z)>3$, are symmetric with
respect to center $[0,0,0,0]$.

It suggests to consider spherical coordinates and gradient of a scalar
function in spherical coordinates:
\begin{align*}
x=\begin{bmatrix}
\rho\cos\phi_1\\\rho\cos\phi_2\sin\phi_1\\\rho\cos\phi_3\sin\phi_1\sin\phi_2\\\rho\sin\phi_3\sin\phi_1\sin\phi_2
\end{bmatrix}
,\quad
\gradient f=\begin{bmatrix}
f\gendiff{\rho}
\\f\gendiff{\phi_1}/\rho\\f\gendiff{\phi_2}/\rho\sin\phi_1\\
f\gendiff{\phi_3} /\rho\sin\phi_1\sin \phi_2
\end{bmatrix}
\end{align*}
where $f\gendiff{\xi}$ stands for $\frac{\diff f}{\diff \xi}$ for sake
of shortness.
Let us suppose for checking the stronger condition at $z$ that
\begin{align*}
\decoupling\biggl(\sum_{i\in M} \lambda_i g_i\biggr)(z)=\begin{bmatrix}\mu_1&\mu_2&\mu_3&\mu_4\end{bmatrix}^t
\end{align*}
then we have to check whether there exists a solution to the polynomial
equations
\begin{align*}
t(z)=&\mu_1(z) \rho c_1+\mu_2(z) \rho c_2 s_1+\mu_3(z) \rho c_3 s_1 s_2+\mu_4(z) \rho s_1 s_2 s_3
\\&\quad+(\rho c_1+\rho s_1 s_2 s_3)^2 (\rho^2 c_1 s_1 s_2 s_3-1)^2
\\&\quad+(\rho c_1+\rho c_2 s_1)^2 (\rho^2 c_1 c_2 s_1-1)^2
\\&\quad+(\rho c_2 s_1+\rho c_3 s_1 s_2)^2 (\rho^2 c_2 s_1^2 c_3 s_2-1)^2
\\&\quad+(\rho c_3 s_1 s_2+\rho s_1 s_2 s_3)^2 (\rho^2 c_3 s_1^2 s_2^2 s_3-1)^2\\
1=&c_i^2+s_i^2,\quad i=1\cdots 3
\end{align*}
in the variables $\rho,c_i,s_i$ related to the modulus and
cosine and sine of the spherical coordinates. 

Notice that only feasible directions are required so that further
conditions would be added to restrict spherical coordinates depending
on $z$ to fit the 4-dimensional hypercube; in particular $\rho^2\in\bigl\{1, 2, 3,
4, 5, 6, 7, 8, 9, 10, 12, 13, 16\bigr\}$.
Clearly, existence of a solution to either system of $4$ non linear
equations in $6$ variables implies the
true satisfaction of the original condition.
To our knowledge, this example overtakes the recent advances
in polynomial systems
\cite{elkadi:inria-00073109} both in dimension and degree
and remains an opened challenging issue for the identity
mapping; furthermore, should a positive answer given, no implicit definition
of the multipliers will be found so that the promise of improving a
not optimal point thanks to the subdifferential flies away.

\section{Concluding remarks}
In this article, we introduce the sign change counting function
as a natural extension of the sign counting function
whose role is prominent in studying various set of matrices.
Its \FT-subdifferential is given and an attempt is made 
to prove a sufficient condition for global optimality at a point.
The failure of this condition, which was specially tailored 
in the case of the identity map $\Rset^n\to\Rset^n$,
raises the question of new sufficient conditions
for lifted maps $\transdir:\Rset^p\to\Rset^n$ with $p>n$
that afford checking for the inverse of some linear system
in $p$ parameters related to lagrangian multipliers
(after projection back into original space $\Rset^n$).
We believe that the sign change counting function 
is interesting in itself and has furthermore,
a fundamental impact in fully splitted spectrum of matrices
as devised by our initial motivation. 
Last, it is worth mentioning, as a quite long term goal,
the design of an optimal sort algorithm associated with
a given vector;
the almost ready made extension of this function
(just think of $J(x)=\bigl\{i\lvert x_i=x_{i+1}\bigr\}$),
would lead to an algorithm
that follows the $\FT$-subdifferential to apply
necessary and sufficient exchanges, on the contrary
to general purpose sort algorithms.

\bibliographystyle{habbrv}

\begin{thebibliography}{10}

\bibitem{MR2148133}
D.~Carrasco-Olivera and F.~Flores-Baz{\'a}n.
\newblock On the representation of approximate subdifferentials for a class of
  generalized convex functions.
\newblock {\em Set-Valued Anal.}, 13(2):151--166, 2005.

\bibitem{elkadi:inria-00073109}
M.~Elkadi and B.~Mourrain.
\newblock {Some Applications of Bezoutians in Effective Algebraic Geometry}.
\newblock Technical Report RR-3572, INRIA, Dec. 1998.

\bibitem{MR1449790}
F.~Flores-Baz{\'a}n.
\newblock On minima of the difference of functions.
\newblock {\em J. Optim. Theory Appl.}, 93(3):525--531, 1997.

\bibitem{ft2013}
D.~Fortin and I.~Tsevendorj.
\newblock {Q}-subdifferential and {Q}-conjugate for global optimality.
\newblock {\em Comp. Math. Phys.}, (1), 2014.
\newblock to appear.

\bibitem{MR2925620}
J.-B. Hiriart-Urruty and H.~Y. Le.
\newblock Convexifying the set of matrices of bounded rank: applications to the
  quasiconvexification and convexification of the rank function.
\newblock {\em Optim. Lett.}, 6(5):841--849, 2012.

\bibitem{MR1704112}
J.-B. Hiriart-Urruty and A.~S. Lewis.
\newblock The {C}larke and {M}ichel-{P}enot subdifferentials of the eigenvalues
  of a symmetric matrix.
\newblock {\em Comput. Optim. Appl.}, 13(1-3):13--23, 1999.
\newblock Computational optimization---a tribute to Olvi Mangasarian, Part II.

\bibitem{MR2324960}
V.~Jeyakumar, A.~M. Rubinov, and Z.~Y. Wu.
\newblock Non-convex quadratic minimization problems with quadratic
  constraints: global optimality conditions.
\newblock {\em Math. Program.}, 110(3, Ser. A):521--541, 2007.

\bibitem{MR3035526}
H.~Y. Le.
\newblock Generalized subdifferentials of the rank function.
\newblock {\em Optim. Lett.}, 7(4):731--743, 2013.

\bibitem{MR1687292}
A.~S. Lewis.
\newblock Nonsmooth analysis of eigenvalues.
\newblock {\em Math. Program.}, 84(1, Ser. A):1--24, 1999.

\bibitem{MR2162512}
A.~S. Lewis and H.~S. Sendov.
\newblock Nonsmooth analysis of singular values. {I}. {T}heory.
\newblock {\em Set-Valued Anal.}, 13(3):213--241, 2005.

\bibitem{MR2162513}
A.~S. Lewis and H.~S. Sendov.
\newblock Nonsmooth analysis of singular values. {II}. {A}pplications.
\newblock {\em Set-Valued Anal.}, 13(3):243--264, 2005.

\bibitem{MR1834382}
A.~Rubinov.
\newblock {\em Abstract convexity and global optimization}, volume~44 of {\em
  Nonconvex Optimization and its Applications}.
\newblock Kluwer Academic Publishers, Dordrecht, 2000.

\bibitem{MR2496428}
A.~M. Rubinov and Z.~Y. Wu.
\newblock Optimality conditions in global optimization and their applications.
\newblock {\em Math. Program.}, 120(1, Ser. B):101--123, 2009.

\bibitem{MR1856798}
I.~Tsevendorj.
\newblock Piecewise-convex maximization problems: global optimality conditions.
\newblock {\em J. Global Optim.}, 21(1):1--14, 2001.

\bibitem{MR2336771}
Z.~Y. Wu, V.~Jeyakumar, and A.~M. Rubinov.
\newblock Sufficient conditions for global optimality of bivalent nonconvex
  quadratic programs with inequality constraints.
\newblock {\em J. Optim. Theory Appl.}, 133(1):123--130, 2007.

\bibitem{MR2602910}
Z.~Y. Wu and A.~M. Rubinov.
\newblock Global optimality conditions for some classes of optimization
  problems.
\newblock {\em J. Optim. Theory Appl.}, 145(1):164--185, 2010.

\end{thebibliography}
\def\cprime{$'$} \def\cftil#1{\ifmmode\setbox7\hbox{$\accent"5E#1$}\else
  \setbox7\hbox{\accent"5E#1}\penalty 10000\relax\fi\raise 1\ht7
  \hbox{\lower1.15ex\hbox to 1\wd7{\hss\accent"7E\hss}}\penalty 10000
  \hskip-1\wd7\penalty 10000\box7}

\end{document}